\date{August 14, 2008}

\documentclass[11pt]{amsart}

\usepackage{amscd}
\usepackage{amsfonts,amssymb,latexsym}
\input xy
\xyoption{all}
\xyoption{arc}
\CompileMatrices

\newcommand{\SA}{{\mathcal{A}}}

\newcommand{\SE}{{\mathcal{E}}}

\newcommand{\SG}{{\mathcal{G}}}
\newcommand{\SH}{{\mathcal{H}}}

\newcommand{\SM}{{\mathcal{M}}}

\newcommand{\SO}{{\mathcal{O}}}

\newcommand{\cS}{{\mathcal{S}}}

\newcommand{\PP}{\mathbb{P}}
\newcommand{\ZZ}{\mathbb{Z}}
\newcommand{\CC}{\mathbb{C}}

\newcommand{\BA}{\mathbb{A}}
\newcommand{\FM}{\mathfrak{M}}

\newcommand{\Ext}{\operatorname{Ext}}

\newcommand{\Hom}{\operatorname{Hom}}

\newcommand{\Pic}{\operatorname{Pic}}

\newcommand{\inj}{\hookrightarrow}
\newcommand{\too}{\longrightarrow}
\newcommand{\rk}{\operatorname{rk}}

\newcommand{\wt}{\widetilde}

\newcommand{\GL}{\operatorname{GL}}
\newcommand{\slr}{\operatorname{SL}_r}
\newcommand{\PGL}{\operatorname{PGL}}
\newcommand{\gitq}{/\!\!/}

\newcommand{\grq}{\operatorname{gr}}

\newtheorem{proposition}{Proposition}[section]
\newtheorem{theorem}[proposition]{Theorem}

\newtheorem{lemma}[proposition]{Lemma}

\newtheorem{corollary}[proposition]{Corollary}

\numberwithin{equation}{section}

\begin{document}

\title[Torelli for moduli of framed bundles]{Torelli
theorem for the moduli space of framed bundles}

\author[I. Biswas]{Indranil Biswas}

\address{School of Mathematics, Tata Institute of Fundamental
Research, Homi Bhabha Road, Bombay 400005, India}

\email{indranil@math.tifr.res.in}

\author[T. G\'omez]{Tom\'as L. G\'omez}

\address{Instituto de Ciencias Matem\'aticas CSIC-UAM-UC3M-UCM,
Serrano 113bis, 28006 Madrid, Spain; and
Facultad de Ciencias Matem\'aticas,
Universidad Complutense de Madrid, 28040 Madrid, Spain}

\email{tomas.gomez@mat.csic.es}

\author[V. Mu\~noz]{Vicente Mu\~noz}

\address{Instituto de Ciencias Matem\'aticas CSIC-UAM-UC3M-UCM,
Serrano 113bis, 28006 Madrid, Spain; and
Facultad de Ciencias Matem\'aticas,
Universidad Complutense de Madrid, 28040 Madrid, Spain}

\email{vicente.munoz@imaff.cfmac.csic.es}



\thanks{Supported by grant MTM2007-63582 from the Spanish
Ministerio de Educaci\'on y Ciencia and grant 200650M066 from
Comunidad Aut\'onoma de Madrid}

\begin{abstract}
   Let $X$ be an irreducible smooth complex projective curve of genus $g>2$,
   and let $x\in X$ be a fixed point. A framed bundle is a pair
   $(E,\varphi)$, where $E$ is a vector bundle over $X$, of rank $r$ and degree $d$, and
   $\varphi:E_x\too \CC^r$ is a non--zero homomorphism.
   There is a notion of (semi)stability for framed bundles depending on a parameter
   $\tau>0$, which gives rise to the moduli space of $\tau$--semistable
framed bundles $\SM^\tau$.
We prove a Torelli theorem for $\SM^\tau$, for $\tau>0$ small enough,
meaning, the isomorphism class of the one--pointed
curve $(X\, ,x)$, and also the integer $r$,
are uniquely determined by the isomorphism class of
the variety $\SM^\tau$.
\end{abstract}

\maketitle

Let $X$ be an
irreducible smooth projective curve
defined over $\mathbb C$. Fix a point $x\,\in\, X$.
Fix a line bundle $\xi$ over $X$, and let $d$ denote its degree.
We consider pairs
of the form $(E,\varphi:E_x\too \CC^r)$, where $E$ is a vector
bundle of fixed rank $r$ and determinant $\xi$,
and $\varphi$ is a $\mathbb C$--linear
homomorphism ($E_x$ is the fiber of $E$ over the point $x$).
This is a particular case of the framed bundles of Huybrechts
and Lehn \cite{HL}. In our situation
the reference sheaf is the torsion sheaf supported at
$x$ with fiber $\CC^r$. In \cite{HL}, the
notion of a semistable framed bundle
is introduced which depends on a real parameter $\tau$,
and the corresponding moduli space is constructed, which is a
complex projective variety.

Let $\tau>0$ be a real number. A pair $(E,\varphi:E_x\too \CC^r)$ is
called $\tau$--\textit{stable} (respectively, $\tau$--\textit{semistable})
if, for all proper subbundles $E'\subset E$ of positive rank,
$$
\frac{\deg E' - \epsilon(E',\varphi)\tau}{\rk E'}
< \frac{\deg E - \tau}{\rk E}
$$
(respectively, $\frac{\deg E' - \epsilon(E',\varphi)\tau}{\rk E'}
\leq \frac{\deg E - \tau}{\rk E}$), where
$$
\epsilon(E',\varphi)=\left\{
\begin{array}{lcl}
1 & \text{if} & \varphi|_{E'_x}\neq 0, \\
0 & \text{if} & \varphi|_{E'_x}=0.
\end{array}
\right .
$$

Let $\SM^{\tau}_{X,x,r,\xi}$
be the moduli space of $\tau$--semistable pairs
of rank $r$
and determinant $\xi$. When the data is
clear from the context, we will also use the shortened
notation $\SM$ instead of $\SM^{\tau}_{X,x,r,\xi}$.

We prove the following Torelli theorem for this moduli space when
$\tau$ is sufficiently small.

\begin{theorem}
\label{thm:main}
Let $X$ be a smooth projective curve of genus $g>2$
and $x\in X$ a point.
Let $r>1$ be an integer and $\xi$ a line
bundle over $X$. Let
$\tau>0$ be a real number with $\tau<\tau(r)$
(cf. Lemma \ref{smalltau}). Let $X',g',x',r',\xi'$ and $\tau'$ be
another set of data satisfying the same conditions. If the
moduli space $\SM^\tau_{X,x,r,\xi}$ is
isomorphic to $\SM^{\tau'}_{X',x',r',\xi'}$,
then there is an isomorphism between $X$ and $X'$ sending
$x$ to $x'$, and also $r=r'$.
\end{theorem}

In other words, if we are given $\SM^{\tau}_{X,x,r,\xi}$ as an abstract variety,
we can recover the curve $X$, the point $x$, and the rank $r$.
We will first prove some facts about the geometry of this
moduli space $\SM\,:=\,
\SM^{\tau}_{X,x,r,\xi}$, and, using these, in the last section we prove
Theorem \ref{thm:main}.

\section{Forgetful morphism}
\label{sec:forget}

The following lemma relates $\tau$--semistability of a framed
bundle with the usual semistability of its underlying vector bundle.

\begin{lemma}
\label{smalltau}
There is a constant $\tau(r)$ that depending only on the rank
$r$ such that for all $\tau\, \in\, (0\, , \tau(r))$ the following
hold:
\begin{enumerate}
\item  $(E,\varphi)$ is $\tau$--semistable $\Rightarrow$ $E$
is semistable.
\item  $E$ is stable $\Rightarrow$ $(E,\varphi)$ is $\tau$--stable.
\item  Any $\tau$--semistable pair is $\tau$--stable.
\end{enumerate}
\end{lemma}

\begin{proof}
The expression $|\epsilon/r'-1/r|$, where $\epsilon=0,1$
and $r'$ is an integer with $0<r'<r$, takes only a finite
number of values, so there is a largest positive number
$\tau(r)$ such that the following holds: if $0<\tau<\tau(r)$, then
\begin{equation}
\label{for1}
0<\tau\Big| \frac{\epsilon}{r'}-\frac{1}{r}\Big|<\frac{1}{r!}\; .
\end{equation}

Assume that $(E,\varphi)$ is $\tau$--semistable but $E$ is not
semistable
as a vector bundle. 
This means that there is a proper subbundle
$E'$ such that
$$
\frac{d}{r}<\frac{d'}{r'}\leq\frac{d}{r}
+\tau\Big(\frac{\epsilon(E,\varphi)}{r'}-\frac{1}{r}\Big) \ ,
$$
where $d=\deg E$ and $d'=\deg E'$.
But this is impossible because the slopes of all subbundles of $E$
are in $(1/r!)\ZZ$. Hence $d'/r'\geq d/r+1/r!$, which contradicts
(\ref{for1}).

Now assume $E$ is stable but $(E,\varphi)$ is not $\tau$--stable.
There is a proper subbundle $E'$ such that
$$
\frac{d}{r}>\frac{d'}{r'}\geq\frac{d}{r}
+\tau\Big(\frac{\epsilon(E',\varphi)}{r'}-\frac{1}{r}\Big)\ ,
$$
but this pair of inequalities again contradict (\ref{for1}).

Finally, if $(E,\varphi)$ is $\tau$--semistable, then for
all proper subbundles $E'$ of $E$,
$$
\frac{d'}{r'}\leq\frac{d}{r}
+\tau\Big(\frac{\epsilon(E',\varphi)}{r'}-\frac{1}{r}\Big) \ .
$$
We note that $\frac{d'}{r'}\not=\frac{d}{r}
+\tau\big(\frac{\epsilon(E',\varphi)}{r'}-\frac{1}{r}\big)$
because $d'/r'$ and
$d/r$ belong to $(1/r!)\ZZ$, and
$\tau\big(\frac{\epsilon(E',\varphi)}{r'}-\frac{1}{r}\big)$ does
not lie in $(1/r!)\ZZ$. Hence any $\tau$--semistable pair is
$\tau$--stable.
This completes the proof of the lemma.
\end{proof}

Henceforth, we will always assume that $0\, <\, \tau\,<\,\tau(r)$.

\medskip

Let $\FM$ denote the moduli space of semistable vector bundles
$E$ over $X$ of rank $r$ with $\bigwedge^rE \,=\, \xi$.
Since the underlying vector bundle of a $\tau$--semistable pair is
semistable, we have a forgetful morphism
\begin{equation}
  \label{eq:forget}
  f:\SM \too \FM
\end{equation}
that sends any framed bundle to the underlying vector bundle.

\begin{lemma}\label{lem:dim-moduli}
If $0<\tau<\tau(r)$, then the moduli space $\SM$ is smooth and
irreducible of dimension $(r^2-1)g$.
\end{lemma}

\begin{proof}
By Lemma \ref{smalltau} (1), there are no strictly semistable pairs.
Hence using \cite[Theorem 4.1]{HL} we conclude that $\SM$ is smooth if
$$
\operatorname{\mathbb{E}xt}^2(E,E\too \CC^r)=0
$$
for all pairs $(E\, ,\varphi)$.
Here we are considering $E$ and $\alpha:E\too \CC^r$
as complexes concentrated in dimensions $0$ and $(0,1)$. The
homomorphism
$\alpha$ is the composition of the restriction to the fiber over
$x$ followed by $\varphi$.
Since $E$ is locally free, the hyper--Ext group is isomorphic to
the hypercohomology
\begin{equation}\label{eq:hyper}
\mathbb{H}^2(E^\vee\otimes E \too E^\vee_x\otimes\CC^r) \, .
\end{equation}
There is a spectral sequence
$$
H^i(\SH^j) \Rightarrow \mathbb{H}^{i+j}(E^\vee\otimes E \too
E^\vee_x\otimes\CC^r)\, ,
$$
where $\SH^j$ is the cohomology of the complex.
Note that $H^2(\SH^0)=0$ because $\dim X=1$, and
$H^1(\SH^1)=0$ because $\SH^1$ is supported at $x$.
Furthermore,
we have $H^0(\SH^2)=0$ because $\SH^2=0$. Therefore, the
hypercohomology group in (\ref{eq:hyper}) is zero
for any pair $(E,\varphi)$. This proves that $\SM$ is smooth.

By Lemma \ref{smalltau} (2), the forgetful morphism
$f$ in (\ref{eq:forget}) is dominant, and furthermore,
it is a projective bundle over the dense open set
$\FM^s$ that parametrizes the stable vector bundles.
Therefore, from the facts
that $\FM$ is irreducible projective and $\SM$ is smooth
projective it follows that $\SM$ is irreducible. Also,
the dimension of $\SM$ is
the sum of $\dim \FM=(r^2-1)(g-1)$ and
the dimension $r^2-1$
of the generic fiber of the forgetful morphism. This completes
the proof of the lemma.
\end{proof}

\section{The $\PGL_r({\mathbb C})$ action on $\SM$}
\label{sec:pglaction}

As before,
let $\FM^s$ be the Zariski open subset of $\FM$ that parameterizes
the stable bundles. The Zariski open subset of $\SM$ that parametrizes
all pairs such that the underlying vector bundle is stable will
be denoted by ${\SM^s}$ (the openness of ${\SM^s}$ follows from
\cite[p. 635, Theorem 2.8(B)]{Ma}). Consider the Cartesian diagram
\begin{equation}\label{de.fs}
\xymatrix{
{\SM^s} \ar[r] \ar[d]_{f_s} & {\SM} \ar[d]^{f}\\
{\FM^s} \ar[r]  & {\FM .}  \\
}
\end{equation}
The morphism $f_s$ defines a projective bundle whose
fiber over any $E$ is the projective space
$\PP(\Hom(E_x,\CC^r))$ that parametrizes all lines in
$\Hom(E_x,\CC^r)\,=\, (E^\vee_x)^{\oplus r}$,
where $E_x$ is the fiber of $E$ over $x$.

If $\FM^s$ admits a universal vector bundle $\SE$, then,
denoting by $\SE_x$ the restriction of $\SE$ to the
slice $\{x\}\times\FM^s$, we have a natural isomorphism
$$
\SM^s \cong \PP(Hom(\SE_x^{},\SO_{\FM^s}^r))\, .
$$
Since automorphisms of a stable vector bundle $E$ act
trivially on $\PP(\Hom(E_x\, ,{\mathbb C}^r))$, even if $\SE_x$ does not
exist, the projective universal
bundle $\PP(Hom(\SE_x^{},\SO_{\FM^s}^r))$
does exist on $\FM^s$.

We will now estimate the codimension
of the complement of $\SM^s$ in $\SM$.

\begin{lemma}\label{lem:codim}
Assume that $g(X)\,>\, 2$. Then the codimension in $\SM$ of
the closed subset $Z$ of pairs $(E,\varphi)$
with $E$ not stable is at least two.
\end{lemma}

\begin{proof}
Clearly, this codimension coincides with the codimension of the
closed subset of strictly semistable vector bundles inside the
moduli space of semistable vector bundles $E$ of rank $r$
and degree $\deg E\, =d$; 
the dimension
of this moduli space is $r^2(g-1)+1$.
So we will calculate the later, and we do this by
refining the computations carried out in Proposition 7.9 of
\cite{BGMMN}.

For a semistable bundle $E$, we
have the Jordan--H\"older filtration given by
 \begin{equation}\label{j-h-filtration}
 0=E_0\subset E_1\subset E_2 \subset \ldots \subset E_l=E\, ,
 \end{equation}
with $Q_i=E_i/E_{i-1}$ stable and
$\mu(Q_i)=\mu(E)$ for $1\leq i \leq l$. The
direct sum
$$\grq(E)\,:=\,\bigoplus_{i=1}^l Q_i$$ is called the {\em graded
vector bundle} associated
to $E$. Combining the isomorphic direct summands we have
$$
\grq(E)\,=\,\bigoplus_{j=1}^k  Q_j^{a_j}\, ,
$$
where $Q_j\not\cong Q_{j'}$ for $j\neq j'$. 

Fix $k\geq 1$, and fix $a_1,\ldots, a_k>0$ with $l=\sum a_j>1$.
Also, for each $j$ with $1\leq j \leq k$,
fix $d_j,m_j$ with $d_j/m_j=\mu(E)$. Consider the space
 $$
\big\{(Q_1,\ldots, Q_k) \in M(m_1,d_1)\times \cdots \times M(m_k,d_k)
\ | \
Q_j\not\cong Q_{j'},
 \ \text{ for } j\neq j' \big\}\, ,
$$
where $M(m_j,d_j)$ is the moduli space of stable vector bundles
over $X$ of rank $m_j$ and degree $d_j$.
Now fix some map 
 $$
 \varpi:\{1,\ldots, l\} \too \{1,\ldots, k\}
 $$
such that $\#\varpi^{-1}(j)=a_j$, for all $1\leq j\leq k$.
This determines the order in which the elements $Q_j$ appear in
the Jordan--H\"older filtration. Then the set of
isomorphism classes of semistable vector bundles
of rank $m$ and degree $d$, whose Jordan--H\"older
filtration (\ref{j-h-filtration}) has
graduation $Q=\bigoplus_{j=1}^k Q_j^{a_j}$,
is covered by the sets $\cS_{\varpi}$ whose elements are the
equivalence classes of extensions of the form
 $$
 0\too E_{i-1}\too E_i\too Q_{\varpi(i)} \too 0,
 $$
where $2\leq i\leq l$, and $E_1=Q_{\varpi(1)}$. Then
  $$
 \dim \cS_{\varpi}\, \leq\,  \sum_{j=1}^k \dim M(m_j,d_j) +
\sum_{i=2}^l (\dim \Ext^1(Q_{\varpi(i)},E_{i-1})-1) \,.
  $$

First recall that $\dim M(m_j,d_j)= m_j^2(g-1) +1$. Second, note that
 $$
 \dim \Ext^1(Q_{\varpi(i)},E_{i-1}) )=-\chi \big(Hom(Q_{\varpi(i)},
E_{i-1})\big)+ \dim \Hom(Q_{\varpi(i)},E_{i-1})\, .
 $$
By Riemann--Roch, and the condition that $\mu(Q_j)=\mu(Q_{j'})$,
we have
\begin{eqnarray*}
&&   -\sum_{i=2}^l \chi \big(Hom(Q_{\varpi(i)},E_{i-1})\big) =
 \sum_i \rk(Q_{\varpi(i)}) \rk(E_{i-1}) (g-1)= \\
&& \hspace{0.5cm} =\frac{r^2-\sum_{i=1}^l \rk(Q_{\varpi(i)})^2}{2} (g-1)=
\frac{r^2-\sum_{j=1}^k a_j m_j^2}{2} (g-1).
\end{eqnarray*}
Also, $\dim\Hom(Q_{\varpi(i)},E_{i-1})$ is at most
the number of times the vector bundle $Q_{\varpi(i)}$ appears in
$\grq(E_{i-1})$. Therefore, a straight--forward induction shows that
$$
\sum_{i=2}^l \dim \Hom(Q_{\varpi(i)},E_{i-1})
\leq \sum_j (1+\ldots +(a_j-1)) =\sum_j \frac{(a_j-1)a_j}{2} \, .
$$
Putting all together,
\begin{eqnarray*}
&&   \dim \cS_{\varpi}\leq  \\
&& \leq
 \sum_j (m_j^2(g-1) +1) + \frac{r^2-\sum_j a_j m_j^2}{2} (g-1) - \sum_j a_j
 +1 +\sum_j \frac{(a_j-1)a_j}{2} \, .
\end{eqnarray*}
Therefore, the codimension is at least
\begin{eqnarray*}
&&    \frac{r^2+\sum_j a_j m_j^2}{2} (g-1)  - \sum_j (m_j^2(g-1) +1) + \sum_j a_j  -\sum_j \frac{(a_j-1)a_j}{2}  =\\
&& \hspace{0.5cm} = \sum_j \left( \frac{a_j^2+ a_j -2}{2} m_j^2 (g-1) - \frac{(a_j-2)(a_j-1)}{2}\right)  + \sum_{j<j'}a_ja_{j'}m_jm_{j'}(g-1)\, .
\end{eqnarray*}
That is, this codimension is bigger than or equal to
\begin{eqnarray*}
&&  \sum_{j=1}^k \frac{a_j-1}{2} ((a_j +2) m_j^2 (g-1) - (a_j-2))   +
\sum_{j<j'}a_ja_{j'}m_jm_{j'}(g-1)  \geq\\
&&\hspace{0.5cm} \geq
  \sum_{j<j'}a_ja_{j'}m_jm_{j'}(g-1)  \ ,\\
\end{eqnarray*}
with strict inequality if not all $a_j=1$. If all $a_j=1$,
then $k>1$ and this number is at least $(r-1)(g-1)$.
If some $a_j>1$, then the worst case is when $k=1$,
in which case $a_j=l\geq 2$ and the above number is
 $$
   \frac{l-1}{2} ((l +2) m_1^2 (g-1) - (l-2))   \geq 2m_1^2 (g-1) \geq 2\ ,
 $$
unless $m_1=1$ and $g=2$. But this case is ruled out by our assumptions.
This completes the proof of the lemma.
\end{proof}

There is an action of $\PGL_r(\mathbb C)$ on the moduli space $\SM$
which is constructed as
follows. Let $[G]\in \PGL_r(\mathbb C)$, and let $G\in \GL_r(\mathbb C)$
be an element that projects to $[G]$. A point in $\SM$ corresponding
to $(E,\varphi:E_x\too \CC^r)$ is
sent to $(E\, ,G\circ \varphi)$.
This is well-defined
because $\epsilon(E'\, ,\varphi)=\epsilon(E'\, ,G\circ \varphi)$ and hence
the $\tau$--semistability is preserved.

Let
$$
T_f\, \subset\, T_{\SM}
$$
be the vertical tangent sheaf for the projection $f$
in (\ref{eq:forget}). In other words, $T_f$ is the kernel of
the differential $df:T_{\SM}\too f^*T_{\FM}$. The map $f$
commutes with the above action
of $\PGL_r(\mathbb C)$ on the moduli space $\SM$.
Hence we get a homomorphism of Lie algebras
\begin{equation}\label{a}
a:\mathfrak{pgl}_r \too 
H^0(\SM,T_f)\, .
\end{equation}

\begin{lemma}
\label{lem:dimtf}
The homomorphism $a$ in (\ref{a}) is an isomorphism.
\end{lemma}

\begin{proof}
Let $\SM^0\subset \SM$ be the open subset parametrizing
all pairs $(E,\varphi)$ such that $E$ is a stable vector
bundle and $\varphi$ is an isomorphism. The action of
$\PGL_r(\mathbb C)$ on $\SM^0$ is evidently free. Hence
$\SM^0$ has the structure of a principal
$\PGL_r(\mathbb C)$--bundle over $\FM^s$.

Therefore, the composition
$$
  \mathfrak{pgl}_r \stackrel{a}{\too} H^0(\SM,T_f)
\too H^0(\SM^0,T_f|_{\SM^0})
$$
is injective. Consequently, the homomorphism
$a$ in (\ref{a}) is injective.

The variety $\FM^s$ admits a covering by Zariski
open subsets $\{U^i_0\}_{i\in I}$ such that for each
$i$ there is an \'etale Galois covering
$$
U^i\, \longrightarrow\, U^i_0
$$
with the property that there is a universal
vector bundle over $U^i\times X$. Note that
any two universal vector bundles over $U^i\times X$
differ by tensoring with a line bundle pulled back from
$U^i$. Let
\begin{equation}\label{U0}
U\,\too\, U_0\, \subset \,\FM^s
\end{equation}
be an \'etale cover
of a nonempty Zariski open subset of $\FM^s$
which admits a universal vector bundle $\SE \too U\times X$.
Therefore, there is a Cartesian diagram
$$
\xymatrix{
{\PP(Hom(\SE_x^{},\SO_{U}^r))} \ar[r] \ar[d]_{f_U} & {\SM^s}
\ar[d]^{f_s}\\
{U} \ar[r] & {\FM^s,}
}
$$
where the map $f_s$ is defined in (\ref{de.fs})
(recall that $\SM^s=f^{-1}(\FM^s)$).
The relative Euler sequence for $f_U$ is
$$
0 \too \SO_{\PP}\too \SO_{\PP}(1)
\otimes {f_U}^* (\SE^{\vee}_x\otimes\SO_{U}^r)
\too T_{f_U} \too 0 \; ,
$$
where $\PP=\PP(Hom(\SE_x^{},\SO_{U}^r))$, and
$T_{f_U}\,\longrightarrow\,
\PP(Hom(\SE_x^{},\SO_{U}^r))$ is the relative tangent
bundle for above the projection $f_U$.
Applying the functor ${f_U}_*$, we get
\begin{equation}\label{s.e.cu}
0 \too \SO_{U}\too
\SE_x\otimes\SO_{U}^{r\vee}
\otimes
\SE_x^\vee\otimes\SO_{U}^r
\too {f_U}_*T_{f_U} \too 0 \; .
\end{equation}

Since any two universal vector bundles over
$U\times X$ differ by tensoring with a line
bundle pulled back from $U$, the short exact
sequence in (\ref{s.e.cu}) is canonical in the
sense that it is independent of the choice
of $\SE$. Therefore, the exact sequence in
(\ref{s.e.cu}) descends to $U_0$ via the map
(\ref{U0}).
Furthermore, these locally defined (in \'etale
topology) short exact sequences patch together
to give a short exact sequence of vector
bundles on $\FM^s$
\begin{equation}
  \label{eq:s1}
  0 \too \SO_{\FM^s}\too
(\SA\oplus \SO_{\FM^s})\otimes\SO_{\FM^s}^{r\vee}
\otimes\SO_{\FM^s}^r
\too {f_s}_*T_{f_s} \too 0 \; ,
\end{equation}
where $\SA$ is the restriction to $\FM^s\times
\{x\}$ of the unique universal adjoint vector bundle
over $\FM^s\times X$, and $T_{f_s}\,\longrightarrow\,
 {\SM^s}$ is the relative tangent bundle for the projection
$f_s$ in (\ref{de.fs}).

Let $P_x\, \longrightarrow\, \FM^s$ be the projective bundle
obtained by restricting to $\FM^s\times\{x\}$ the unique
universal projective bundle over $\FM^s\times X$. The
adjoint bundle associated to $P_x$ 
coincides with the vector bundle $\SA$. Since
the $\text{PGL}_r(\mathbb C)$--bundle defined by $P_x$ is stable,
the vector bundle $\SA$ does not have any
nonzero global sections (cf. \cite{BBN,BG}).
Therefore, taking global sections in (\ref{eq:s1}), we have
the short exact sequence
$$
0 \too \CC \too \mathfrak{gl}_r
\too
H^0(\SM^s,T_{f_s})\too 0 \; .
$$
It follows that $h^0(\SM^s,T_{f_s})
=r^2-1$. From Lemma \ref{lem:codim} we know
that $\SM^s$ is dense in $\SM$.
Hence the restriction homomorphism $H^0(\SM,T_{f})\too
H^0(\SM^s,T_{f_s}) $
is injective. Therefore, we have injective homomorphisms
$$
\mathfrak{pgl}_r \stackrel{a}\inj
H^0(\SM,T_{f})\stackrel{b}\inj
H^0(\SM^s,T_{f_s})\, ,
$$
and the dimensions of the first and last spaces coincide.
Consequently, both $a$ and $b$ are isomorphisms.
\end{proof}

\begin{lemma}
\label{lem:tft}
  There is an isomorphism
$$
H^0(\SM,T_f) = H^0(\SM,T_{\SM})\, .
$$
\end{lemma}

\begin{proof}
Since the map $f_s$ is proper (see (\ref{de.fs})),
we have
$$
H^0(\SM^s,f_s^*T_{\FM^s})=H^0(\FM^s,T_{\FM^s})\, .
$$
It
is known that $H^0(\FM^s,T_{\FM^s})\,=\, 0$
(in \cite{H} this is proved for rank two, but the proof
works in general).
This, together with the exactness of the sequence
$$
0 \too H^0(\SM^s,T_{f_s}) \stackrel{j}{\too} H^0(\SM^s,T_{\SM^s}) \too H^0(\SM^s,f_s^*T_{\FM^s})
$$
implies that $j$ is an isomorphism.
In the proof of Lemma \ref{lem:dimtf} we have seen
that
$H^0(\SM,T_{f})\cong H^0(\SM^s,T_{f_s})$.
On the other hand,
$$
H^0(\SM,T_{\SM}) \hookrightarrow H^0(\SM^s,T_{\SM^s})
\cong H^0(\SM,T_f)
$$
has a left inverse $$H^0(\SM,T_f) \longrightarrow  H^0(\SM,T_{\SM})$$
given by the inclusion of sheaves $T_f\hookrightarrow T_{\SM}$. So
$H^0(\SM,T_{\SM})\cong H^0(\SM,T_f)$, and the Lemma is proved.
\end{proof}

Combining Lemmas \ref{lem:tft} and \ref{lem:dimtf} we
get the following corollary.

\begin{corollary}
\label{cor:dimtm}
  $\dim H^0(\SM,T_\SM)=r^2-1$.
\end{corollary}

\begin{proposition}
\label{prop:uniqueaction}
There is a unique nontrivial action of $\PGL_r(\mathbb C)$ on $\SM$
up to a group automorphism of $\PGL_r(\mathbb C)$.
\end{proposition}

\begin{proof}
  An effective action gives in injection
$i:\mathfrak{pgl}_r\inj H^0(\SM,T_\SM)$. Using
Corollary \ref{cor:dimtm} we conclude that this
homomorphism $i$ is an isomorphism.
Note that $i$ and $a$ are isomorphisms of Lie algebras,
therefore $i\circ a^{-1}$ is an automorphism of $\mathfrak{pgl}_r$
which comes from a group automorphism of $\PGL_r(\mathbb C)$.
\end{proof}

\section{The quotient of the action}

In the previous section we have seen that, up to an isomorphism,
there is a unique action of $\PGL_r(\mathbb C)$ on $\SM$.
In this section
we will see that the GIT quotient of $\SM$ by the action
of $\PGL_r(\mathbb C)$ 
is $\FM$.
Before taking the GIT quotient, we have to choose a linearized
polarization, but we will show that, in our case,
the quotient does
not depend on this choice.

Recall that $\Pic \FM = \ZZ$ (cf. \cite{DN}).
The following Lemma calculates $\Pic \SM$.

\begin{lemma}
\label{lem:pic}
  If $g>2$, then there is a short exact sequence
$$
1 \too \Pic \FM \too \Pic \SM \too \ZZ  \too 1 \; ,
$$
where the first map is pull-back of line bundles
from $\FM$, and the second map is restriction to
the generic fiber.
\end{lemma}

\begin{proof}
The open set $\SM^s$ is a projective bundle over
$\FM^s$. Consequently, we have a short
exact sequence
$$
1 \too \Pic \FM^s \too \Pic \SM^s \too \ZZ  \too 1
$$
On the one hand, we have $\Pic \FM^s=\Pic \FM$, and on the other
hand, we have $\Pic \SM = \Pic \SM^s$ by Lemma \ref{lem:codim}.
Hence the lemma follows.
\end{proof}

To do the GIT quotient, we will replace the group
$\PGL_r(\mathbb C)$ by $\slr(\mathbb C)$. This is justified because a $\PGL_r(\mathbb C)$ action
induces an $\slr(\mathbb C)$ action, an $\slr(\mathbb C)$ linearization on
a line bundle $L$ induces a $\PGL_r(\mathbb C)$ linearization on $L^r$,
and both quotients coincide.

On $X\times \SM$ there is a vector bundle $\SE$ and
a homomorphism $$\wt \varphi: \SE_x \too p^*_X \SO^r_{x}$$
(where $\SE_x=\SE|_{x \times \SM}$)
such that $(\SE,\wt \varphi)$ is a universal pair
(cf. \cite[Theorem 0.1]{HL}). The
determinant of $\wt \varphi$ defines a section
of the line bundle
$\wedge^r \SE_x^\vee \otimes \wedge^r \CC^r$ on $\SM$.
Note that this line bundle admits a natural
linearization of the $\slr(\mathbb C)$--action. Furthermore,
the section $\det \wt \varphi$ is $\slr(\mathbb C)$--invariant with
respect to this linearization. We also note that
this linearization is unique up to an isomorphism. Indeed,
any two linearizations on the same line bundle
differ by a character of the group, but $\slr(\mathbb C)$ being
semisimple does not admit any nontrivial character.

\begin{lemma}
\label{lem:slrss}
  Let $(E,\varphi)$ be the $\tau$--stable pair corresponding
to a point in $\SM$.
This point is $\slr(\mathbb C)$--semistable with respect to
any linearized polarization
if and only if
$\varphi$ is an isomorphism between $E_x$ and $\CC^r$.
\end{lemma}

\begin{proof}
Let $L$ be an $\slr(\mathbb C)$--linearized polarization on $\SM$.
The $\slr(\mathbb C)$ action is $f$--invariant, so we may
replace $\SM$ by the fiber of $f$ containing the
point, and $L$ by its restriction to the fiber.
By Lemma \ref{lem:pic}, the group of line bundles
on a fiber of $f$ which are restrictions of line bundles
on $\SM$ is isomorphic to $\mathbb Z$.
Therefore, we may assume that the
restriction, to a fiber of $f$,  of the
$\slr(\mathbb C)$--linearized line bundle $L$
is isomorphic to the restriction of the line bundle
$\wedge^r \SE_x^\vee \otimes \wedge^r \CC^r$,
introduced above, equipped with the natural linearization.

  The $\slr(\mathbb C)$--invariant section $\det \wt \varphi$ is nonzero
when $\varphi$ is an isomorphism, and this proves that the
point is $\slr(\mathbb C)$-semistable.

On the other hand, if $\varphi$ is not an isomorphism, then
fix a decomposition $\CC^r=V_1\oplus V_2$, where
$V_1$ is the image of $\varphi$ and $V_2$ is a linear
complement. Let $\lambda(t)$ be the one parameter subgroup of $\slr(\mathbb C)$
which acts as $t^{\dim V_2}$ on $V_1$ and $t^{-\dim V_1}$ on $V_2$.
It follows that
$$
\lim_{t\too 0} \varphi\cdot \lambda(t)=
\lim_{t\too 0} \varphi \; t^{\dim V_2}   =
0\, .
$$

We note that any section $s$ of $L$ can be thought as
a homogeneous function. If this function is $\slr(\mathbb C)$--invariant,
$$
s(\varphi)=s(\lim_{t\too 0} \varphi\cdot \lambda(t))=s(0)=0 \; .
$$
This
proves that $(E,\varphi)$ is $\slr(\mathbb C)$--unstable when $\varphi$ is not
an isomorphism.
\end{proof}

\begin{proposition}
  \label{prop:git}
The GIT quotient of $\SM$ by the action of $\PGL_r(\mathbb C)$, with
respect to any linearized polarization, is isomorphic to
$\FM$. Furthermore, the quotient morphism, which in principle is only
defined on the $\PGL_r(\mathbb C)$--semistable points, extends uniquely to
the entire $\SM$, and this extension coincides with the forgetful
morphism $f$ defined in ~(\ref{eq:forget}).
\end{proposition}

\begin{proof}
By Lemma \ref{lem:slrss}, the open set $\SM^{ss}$
of $\slr(\mathbb C)$--semistable points coincides with the set 
of
pairs $(E,\varphi)$ such that $\varphi$ is an isomorphism.

Since $f:\SM\too \FM$ is $\PGL_r(\mathbb C)$--invariant, its
restriction to $\SM^{ss}$ factors
as
$$
\SM^{ss}\too \SM\gitq\PGL_r(\mathbb C)
\stackrel{\widetilde g}\too \FM\, ,
$$
where $\SM\gitq\PGL_r(\mathbb C)$ is the GIT quotient.
We want to show that $\widetilde g$
is actually an isomorphism.

  The open subset $\SM^{ss}\cap\SM^s$ coincides with the 
  set $\SM^0$ defined in Lemma \ref{lem:dimtf}. It is a fibration
over $\FM^s$ whose fiber over the point
corresponding to a stable vector bundle $E$
is $\PP(\operatorname{Iso}(E_x,\CC^r))$. The variety
$\FM$ is known to be normal. Therefore,
the morphism $\widetilde g$ is an isomorphism over $\FM^s$.
Hence to show that $\widetilde g$ is an isomorphism
it is sufficient to show that the fiber of $\widetilde g$ over
a strictly semistable vector bundle is just one point.

Let $p\in \FM$ be a point corresponding to a
polystable vector bundle $E_p$ which is not
stable. The fiber $f^{-1}(p)$ is
the set of all $\tau$--stable pairs
$(E,\varphi)$, where $E$ is $S$--equivalent to $E_p$ and $\varphi$
is an isomorphism.  Two pairs $(E,\varphi)$ and
$(E',\varphi')$ are in the same orbit if and only if
$E$ is isomorphic to $E'$.

Let $E_1$ be a vector bundle $S$--equivalent to $E_p$.
There is a family of vector bundles
parameterized by $\BA^1$, with $E_0\cong E_p$
such that $E_t\cong E_1$ when $t\neq 0$. This family is given
by a vector bundle on $X\times \BA^1$, whose restriction to
$\{x\}\times \BA^1$ is obviously trivial. We note that choosing a
trivialization amounts to giving a family of
pairs $(E_{\BA^1},\varphi_{\BA^1})$ such that, for all $t\in \BA^1$,
$\varphi_t$ is an isomorphism, and hence
$(E_t,\varphi_t)$ is $\slr(\mathbb C)$--semistable
(cf. Lemma \ref{prop:git}).
If $t\neq 0$, the pair $(E_t,\varphi_t)$ lies in the
$\slr$--orbit corresponding to $E_1$,
and if $t=0$, then it lies
in the one corresponding to $E_0$,
So, by continuity, both orbits are mapped to the same
point in $\SM\gitq \PGL_r(\mathbb C)$, proving that $\widetilde g$
is an isomorphism.

Finally, since each point of $\SM$ is in the closure of a
unique fiber of $\SM^{ss}\too \SM\gitq\PGL_r(\mathbb C)$, the
quotient map can be extended to the whole of $\SM$, and
it coincides with the forgetful morphism $f$.
\end{proof}

Combining Proposition \ref{prop:git} with Proposition
\ref{prop:uniqueaction} we conclude that the isomorphism
class of the variety $\SM$ uniquely determines the
isomorphism class of the variety $\FM$. Therefore,
the Torelli theorem for $\FM$,
\cite[Theorem E, p. 229]{KP}, gives the following corollary.

\begin{corollary}\label{tor.SM}
The isomorphism class of the variety $\SM$ uniquely
determines the isomorphism class of the curve $X$.
\end{corollary}

\section{Restriction to Hecke cycle}

The open subset $\SM^s$ of $\SM$ lying over the stable
locus $\FM^s$ is a projective bundle
\begin{equation}\label{pxd}
\PP_x\, :=\, \SM^s\, \longrightarrow\, \FM^s
\end{equation}
whose fiber
over a point corresponding to the stable vector bundle
$E$ is canonically isomorphic to
$\PP(\SE_x^\vee\otimes \SO_{\FM^s}^r)$.
In this section we calculate the restriction of this projective
bundle to certain subvarieties which are called Hecke cycles.

Let $l$ and $m$ be integers. A vector bundle $E$ is called
$(l,m)$-stable if for all proper subbundles
$E'$, the following is satisfied:
$$
\frac{\deg E' +l}{\rk E'}
<
\frac{\deg E +l-m}{\rk E} \; .
$$
Let $y\in X$ be a point of the curve. If $(g,r,d)$
is different
from $(2,2,\textup{even})$, then there exist a $(0,1)$-stable
vector bundle $F$ of
rank $r$ and determinant $\xi\otimes {\mathcal O}_X(y)$
(cf. \cite[Proposition 5.4]{NR}).
On $\PP(F_y^\vee)\times X$ there is a short exact sequence
\begin{equation}
  \label{eq:hecke}
  0 \too \SG \too p^*_X F \too \iota_*
\SO_{\PP(F_y^\vee)}(1)\too 0\, ,
\end{equation}
where $\iota\,:\, \PP(F_y^\vee)\, \longrightarrow\,
\PP(F_y^\vee)\times X$ is the map defined by
$z\, \longmapsto\, (z\, , y)$.

The vector bundle $\SG$ is a family of stable bundles because $F$ is
$(0,1)$-stable (cf. \cite[Lemma 5.5]{NR}), and the classifying
morphism $\PP(F_y^\vee)\too \FM$ is an embedding (cf. \cite[Lemma
5.9]{NR}). Its image, which we identify with $\PP(F_y^\vee)$, is
called a \textit{Hecke cycle}.

\begin{lemma}
\label{lem:hecke}
The restriction of the projective fibration $\PP_x\too \FM^s$
(see (\ref{pxd})) to the Hecke cycle
$\PP(F_y^\vee)$ is non-trivial if and only if $x=y$.
\end{lemma}

\begin{proof}
By the universal property of the classifying morphism,
the restriction of $\PP_x$ to the Hecke cycle is
$\PP(\SG|_{\PP(F_y^\vee)\times\{x\}}^\vee \otimes \SO^r)$.
If $x\neq y$, the restriction of (\ref{eq:hecke}) to
$\PP(F_y^\vee)\times\{x\}$ is
$$
0\too \SG|_{\PP(F_y^\vee)\times\{x\}} \stackrel{\cong}{\too} F_y\otimes
\SO_{\PP(F_y^\vee)} \too 0
$$
and hence the projective fibration is trivial on the Hecke cycle.
On the other hand, if $x=y$, then the restriction
of (\ref{eq:hecke}) to $\PP(F_y^\vee)\times \{x\}$ is
\begin{equation}\label{ex.l}
0 \too \SO_{\PP(F_y^\vee)}(1) \too \SG|_{\PP(F_y^\vee)\times\{x\}}
\too F_x\otimes \SO_{\PP(F_y^\vee)}
\too \SO_{\PP(F_y^\vee)}(1) \too 0\, .
\end{equation}
To prove that the line bundle in the left of
(\ref{ex.l}) is indeed
$\SO_{\PP(F_y^\vee)}(1)$, note that
$\deg (\SG|_{\PP(F_y^\vee)\times\{x\}})=0$ for $x\neq y$
because in that case $\SG|_{\PP(F_y^\vee)\times\{x\}}$ is a trivial
vector bundle. Hence by continuity, we have
$\deg (\SG|_{\PP(F_y^\vee)\times\{y\}})=0$, hence the
line bundle in the left of (\ref{ex.l}) is $\SO_{\PP(F_y^\vee)}(1)$.

We will show that
$\PP(\SG|_{\PP(F_y^\vee)\times\{x\}})$ is not trivial for $x=y$.
To prove this by contradiction,
assume that  $\PP(\SG|_{\PP(F_y^\vee)\times\{x\}})$ is trivial.
Consequently,
$\SG|_{\PP(F_y^\vee)\times\{x\}}$ is the tensor product of
the trivial vector bundle and
a line bundle. On the other hand, the
exact sequence in (\ref{ex.l}) shows that
$\SG|_{\PP(F_y^\vee)\times\{x\}}$
has degree zero and admits a line subbundle of degree one.
Therefore, $\SG|_{\PP(F_y^\vee)\times\{x\}}$ is
not the tensor product of the trivial vector bundle with
a line bundle.

Hence $\PP(\SG|_{\PP(F_y^\vee)\times\{x\}})$ is not trivial for $x=y$.
This completes the proof of the lemma.
\end{proof}

In \cite{KP}, the group of all automorphisms of the variety
$\FM$ is described (see \cite[p. 227, Theorem A]{KP}). We
recall that $\text{Aut}(\FM)$ is generated
by all automorphisms of the following three type:
\begin{itemize}
\item $E\, \longmapsto\, E\otimes L$, where $L$ is
a line bundle of order $r$.

\item $E\, \longmapsto\, E^*\otimes L$, where $L$
is a fixed line bundle.

\item $E\, \longmapsto\, L\otimes \sigma^*E$, where
$\sigma$ is an automorphism of $X$.
\end{itemize}

We note that tensoring of a vector bundles by a line
bundle does not alter the projective bundle. Also,
a projective bundle is trivial if and only if the dual
projective bundle is trivial.
Therefore, Lemma \ref{lem:hecke} has the following
corollary:

\begin{corollary}\label{cor.hec}
Let $\tau\, :\, \FM\, \longrightarrow\, \FM$ be an automorphism
such that the restriction of the projective fibration
$\tau^*\PP_x\too \FM^s$ to the Hecke cycle
$\PP(F_y^\vee)$ is non-trivial. Then there is an automorphism
$\sigma$ of $X$ such that $\sigma(x)\, =\, y$.
\end{corollary}

\section{Proof of Theorem \ref{thm:main}}
\label{sec:mainthm}

We start with $\SM$ as an abstract variety.
Since $\dim H^0(\SM,T_\SM)=r^2-1$ (see Corollary
\ref{cor:dimtm}),
we recover $r$, and since the dimension of $\SM$
is $(r^2-1)g$ (see Lemma \ref{lem:dim-moduli}),
we recover $g$.

Up to an automorphism of $\PGL_r(\mathbb C)$,
there is a unique action of $\PGL_r(\mathbb C)$ on $\SM$
(see Proposition \ref{prop:uniqueaction}),
so we recover the action.

Choose any linearized polarization of $\SM$. The GIT quotient
for the action of $\PGL_r(\mathbb C)$ gives a morphism defined on the
semistable points, but
Proposition \ref{prop:git} shows that this can be extended uniquely
to a morphism on $\SM$, and that the image is isomorphic to $\FM$,
where $\FM$ is
the moduli space of semistable vector bundles of rank $r$ and
fixed determinant $\xi$.
Therefore, using the Torelli theorem for the moduli space of
semistable vector bundles, we recover the curve $X$
(see Corollary \ref{tor.SM}).

At this point we have a projective scheme $\SM$ with a
morphism $f:\SM\too \FM$, together with an
action of $\PGL_r(\mathbb C)$ on $\SM$ that commutes with $f$.
The open subset $\SM^s$ of $\SM$ lying over the stable
locus $\FM^s$ is a projective bundle $\PP_x$ whose fiber
over a point corresponding to the stable vector bundle
$E$ is canonically isomorphic to
$\PP(\SE_x^\vee\otimes \SO_{\FM^s}^r)$.

In order to
recover this point $x$, it suffices to show that the projective
bundles $\PP_x\too \FM^s$
and $\PP_y\too \FM^s$ are not isomorphic if
$x\neq \sigma(y)$ for all $\sigma\,\in\, \text{Aut}(X)$.
This follows from Corollary \ref{cor.hec} and
Lemma \ref{lem:hecke}. This finishes the proof
of Theorem \ref{thm:main}.

\medskip

\noindent {\bf Acknowledgments.} We would like to thank Marina Logares for
useful discussions.

\end{document}